\documentclass[11pt, twoside]{article}
\usepackage{latexsym}
\usepackage{amsmath}
\usepackage{amssymb}
\usepackage[all]{xy}
\usepackage{amsfonts}
\usepackage{verbatim}
\usepackage{amsthm}
\usepackage{mathrsfs}
\usepackage{epsfig}
\usepackage{xy}
\usepackage{array}
\usepackage{stmaryrd}
\usepackage{graphicx,color}
\usepackage{xcolor}
\usepackage{numname}
\usepackage[colorlinks=true,linkcolor=blue,citecolor=blue]{hyperref}
\usepackage{tikz}
\usetikzlibrary{arrows,calc}
\usepackage{etex}
\usepackage{mathdots}
\usepackage{float}
\usepackage{graphics}
\usepackage{pdflscape}
\usepackage{tikz-cd}
\usepackage[T1]{fontenc} %Name
\usepackage{lmodern} %Name

\usepackage{anysize,hyperref}
\input xypic
\xyoption{all}
\usepackage{enumitem}
\usepackage[perpage,symbol]{footmisc}
\usepackage{setspace}
\def\P{\mathcal{P}}
\def\I{\mathcal{I}}
\def\J{\mathcal{J}}

\def\C{\mathcal{C}}
\def\X{\mathcal{X}}
\def\Y{\mathcal{Y}}
\def\E{\mathbb{E}}

\def\S{\mathcal{S}}
\def\T{\mathcal{T}}

\def\s{\mathfrak{s}}

\def\Ab{\mathit{Ab}}

\def\dr{\ar@{->}[r]}

\def\ker{\mathsf{Ker}\hspace{.01in}}

\def\im{\mathrm{Im}\hspace{.01in}}

\def\Hom{\mbox{Hom}}
\newcommand{\ra}{\rightarrow}
\newcommand{\lra}{\longrightarrow}
\newcommand{\dra}{\dashrightarrow}

\newcommand{\Cone}{\operatorname{Cone}\nolimits}
\newcommand{\CoCone}{\operatorname{CoCone}\nolimits}

\newcommand{\mv}[2]{\left(\begin{smallmatrix}
		#1 \\
		#2
	\end{smallmatrix}\right)}

\usepackage{bm}
\begin{document}
	\baselineskip=15pt
	\title{\Large{\bf Ideal $n$-cotorsion pairs in Frobenius extriangulated categories}
		\footnotetext{Panyue Zhou is supported by the National Natural Science Foundation of China (Grant No. 12371034) and by the Scientific Research Fund of Hunan Provincial Education Department (Grant No. 24A0221).}}
	\medskip
	\author{Yixia Zhang and Panyue Zhou}
	
	\date{}
	
	\maketitle
	\def\blue{\color{blue}}
	\def\red{\color{red}}

	\newtheorem{theorem}{Theorem}[section]
	\newtheorem{lemma}[theorem]{Lemma}
	\newtheorem{corollary}[theorem]{Corollary}
	\newtheorem{proposition}[theorem]{Proposition}
	\newtheorem{conjecture}{Conjecture}
	\theoremstyle{definition}
	\newtheorem{definition}[theorem]{Definition}
	\newtheorem{question}[theorem]{Question}
	\newtheorem{remark}[theorem]{Remark}
	\newtheorem{remark*}[]{Remark}
	\newtheorem{example}[theorem]{Example}
	\newtheorem{example*}[]{Example}
	\newtheorem{condition}[theorem]{Condition}
	\newtheorem{condition*}[]{Condition}
	\newtheorem{construction}[theorem]{Construction}
	\newtheorem{construction*}[]{Construction}
	\newtheorem{fact}[theorem]{Fact}
	\newtheorem{assumption}[theorem]{Assumption}
	\newtheorem{assumption*}[]{Assumption}	

	\baselineskip=17pt
	\parindent=0.5cm
	\vspace{-6mm}

	\begin{abstract}
\begin{spacing}{1.2}
Motivated by the correspondence between ideal cotorsion pairs in Frobenius exact categories and those in their stable categories, we introduce the notion of an ideal $n$-cotorsion pair in an extriangulated category. We study the relationship between ideal $n$-cotorsion pairs in a Frobenius extriangulated category $\mathcal C$ and those in its stable category $\underline{\mathcal C}=\mathcal C/\omega$. Our main result shows that $(\mathcal I,\mathcal J)$ is an ideal $n$-cotorsion pair in $\mathcal C$ if and only if $(\mathcal I/\omega,\mathcal J/\omega)$ is an ideal $n$-cotorsion pair in $\underline{\mathcal C}$. This provides a bridge between higher ideal approximation theory in Frobenius extriangulated categories and its counterpart in their stable categories. Additionally, in Krull--Schmidt exact categories, we establish a bijective correspondence between complete cotorsion pairs and complete ideal cotorsion pairs, answering a question of Fu, Guil Asensio, Herzog and Torrecillas.
	
\vspace{2mm}

\textbf{\hspace{-5.5mm}Keywords:} extriangulated category; approximation theory; ideal cotorsion pair;
$n$-cotorsion pair; stable category\\[0.1cm]
		\end{spacing}
	\end{abstract}

\section{Introduction}

Approximation theory is a fundamental tool in homological algebra and representation theory. Its basic idea is to study complicated objects by means of relatively simple ones. It has been widely used to construct resolutions, define derived functors and investigate homological dimensions in module categories, abelian categories and exact categories. The notion of cotorsion pairs, introduced by Salce \cite{S} in the category of abelian groups, provides a central framework for classical approximation theory. It was later developed in various settings, including sheaf categories \cite{EERO}, Grothendieck categories \cite{G,M} and exact categories \cite{SS}. In particular, complete cotorsion pairs connect precovers, preenvelopes and special approximations in a unified way.

Classical approximation theory is mainly formulated in terms of objects and subcategories. However, many approximation phenomena are naturally governed by morphisms rather than by objects alone. To address this point, Fu, Guil Asensio, Herzog and Torrecillas \cite{FAHT} introduced ideal approximation theory, replacing subcategories by ideals of morphisms and developing the notion of ideal cotorsion pairs. This framework extends classical approximation theory to the morphism level and establishes a close relation between special precovering ideals and complete ideal cotorsion pairs in exact categories. Breaz and Modoi \cite{BM} further developed ideal approximation theory in triangulated categories. One striking feature of the triangulated setting is that every precovering or preenveloping ideal is automatically special, while this is no longer true in general exact categories. This contrast indicates that ideal approximation behaves differently in exact and triangulated contexts, and therefore motivates the study of intermediate and unifying settings.

Frobenius exact categories provide such a bridge. They are exact categories with enough projective and enough injective objects, and the projective objects coincide with the injective objects. If $\omega$ denotes the full subcategory of projective-injective objects, then the stable category $\mathcal C/\omega$ carries a canonical triangulated structure. Thus it is natural to compare ideal approximation theory in a Frobenius category with that in its stable category. Recently, Sun, Tan, Wang and Zhu \cite{STWZ} proved that, under suitable assumptions on ideals containing projective objects, the precovering property, the special precovering property and the completeness of ideal cotorsion pairs in a Frobenius category are equivalent to the corresponding properties in its stable category.
On the other hand, higher homological methods have led to higher versions of cotorsion theory. Mao and Ding \cite{MD} studied cotorsion dimension and higher-order orthogonality via the functor ${\rm Ext}^{n+1}(-,-)$. Further developments include Gorenstein $n$-cotorsion modules, Gorenstein $n$-flat covers and $n$-cotorsion envelopes over $n$-coherent rings \cite{HMP}.
The notion of $n$-cotorsion pairs was introduced in \cite{HMP} as a higher analogue of complete cotorsion pairs, motivated by approximation phenomena in Gorenstein homological algebra. Its basic idea is to replace the usual first-extension orthogonality by higher extension vanishing up to degree $n$, together with suitable finite resolution or coresolution conditions. This framework recovers complete cotorsion pairs when $n=1$, and connects higher approximation theory with Gorenstein modules, chain complexes, and cluster tilting theory \cite{HMP}. It was later extended to extriangulated categories, where $n$-cotorsion pairs of the form $(\mathcal X,\mathcal X)$ correspond to $(n+1)$-cluster tilting subcategories \cite{HZ}. More recent developments study mutation of $n$-cotorsion pairs in triangulated and extriangulated settings, showing that such pairs are stable under mutation and admit geometric interpretations in higher cluster categories \cite{CZ,CLZ}.

Extriangulated categories, introduced as a common generalization of exact categories and triangulated categories, provide a natural framework in which both ideal approximation theory and higher cotorsion theory can be considered simultaneously. Motivated by the above developments, we introduce the notion of an ideal $n$-cotorsion pair in an extriangulated category. This notion combines the morphism-level viewpoint of ideal cotorsion pairs with the higher orthogonality underlying $n$-cotorsion pairs.
The main purpose of this paper is to compare ideal $n$-cotorsion pairs in a Frobenius extriangulated category with those in its stable category. Let $\mathcal C$ be a Frobenius extriangulated category, and let $\omega$ be the full subcategory of projective-injective objects. We write $\underline{\mathcal C}=\mathcal C/\omega$ for the stable category. For ideals $\mathcal I$ and $\mathcal J$ of $\mathcal C$, we denote by $\mathcal I/\omega$ and $\mathcal J/\omega$ their induced ideals in $\underline{\mathcal C}$. Our main result states that $(\mathcal I,\mathcal J)$ is an ideal $n$-cotorsion pair in $\mathcal C$ if and only if $(\mathcal I/\omega,\mathcal J/\omega)$ is an ideal $n$-cotorsion pair in $\underline{\mathcal C}$.
This theorem shows that the higher ideal cotorsion structure of a Frobenius extriangulated category can be detected after passing to its stable category. In particular, when $\mathcal C$ is a Frobenius exact category and $n=1$, our result recovers the correspondence studied in \cite{STWZ}. Thus the present work extends the relation between ideal approximation theory in Frobenius categories and stable categories to the higher ideal setting.

Inspired by the work of Fu, Guil Asensio, Herzog and Torrecillas \cite{FAHT}, we study the relationship between complete cotorsion pairs and complete ideal cotorsion pairs in extriangulated categories. In \cite[Question 28]{FAHT}, the authors asked whether a complete ideal cotorsion pair whose two ideals are object ideals necessarily arises from a complete cotorsion pair. We provide a positive answer to this question in the Krull-Schmidt setting. Specifically, we prove two main results: Theorem~\ref{main11} shows that every complete cotorsion pair $(\mathcal{X},\mathcal{Y})$ in an extriangulated category gives rise to a complete ideal cotorsion pair $([\mathcal{X}],[\mathcal{Y}])$; conversely, Theorem~\ref{main12} establishes that under the Krull-Schmidt assumption, every complete ideal cotorsion pair of the form $([\mathcal{X}],[\mathcal{Y}])$ comes from a complete cotorsion pair $(\mathcal{X},\mathcal{Y})$. As a consequence, in Krull--Schmidt exact categories, we obtain an affirmative answer to Question~28 of \cite{FAHT}. A related result for Frobenius exact categories was given in \cite[Corollary 1.3]{STWZ}; however, our approach is direct and does not rely on stable category techniques, thus applying to arbitrary Krull--Schmidt exact categories.

The paper is organized as follows. In Section 2,
 we recall the definition of extriangulated categories and some basic facts.
  In Section 3, we answer a question of Fu, Guil Asensio, Herzog and Torrecillas \cite[Question 28]{FAHT} by establishing, in the Krull-Schmidt setting, a bijective correspondence between complete cotorsion pairs and complete ideal cotorsion pairs (see Theorems~\ref{main11} and~\ref{main12}).
 In Section 4, we introduce ideal $n$-cotorsion pairs in extriangulated categories, establish their elementary properties, prove the main correspondence theorem (see Theorem \ref{main result of ideal n-cotorsion pair}), and give some applications.

\section{Preliminaries}
\setcounter{equation}{0}

In this section,  we omit the detailed axioms and definitions on extriangulated categories, and only give terminologies and properties which we shall need later. For the precise definition and the detailed properties of extriangulated categories, we refer to \cite[Section 2]{NP}.

Let $\mathcal{C}$ be an additive category.  For objects $X,Y \in \mathcal{C}$, we denote by $\C(X,Y)$ or ${\rm Hom}_{\C}(X,Y)$ the set of morphisms from $X$ to $Y$ in $\mathcal{C}$.

\subsection{Extriangulated categories}

We briefly recall the definition and some basic properties of extriangulated categories from \cite{NP}.
Let $\mathcal{C}$ be an additive category equipped with an additive bifunctor
$$
\mathbb{E}\colon \mathcal{C}^{\rm op}\times \mathcal{C}\longrightarrow \mathrm{Ab},
$$
where $\mathrm{Ab}$ denotes the category of abelian groups. For objects $A,C\in\mathcal{C}$, an element $\delta\in \mathbb{E}(C,A)$ is called an \emph{$\mathbb{E}$-extension}.

The equivalence classes $[A \overset{x}{\lra} B \overset{y}{\lra} C]=[A \overset{x^{\prime}}{\lra} B^{\prime} \overset{y^{\prime}}{\lra} C]$ means that
there exists an isomorphism $b\in\C(B,B^{\prime})$ which makes the following diagram commutative.
\[
\xy
(-16,0)*+{A}="0";
(3,0)*+{}="1";
(0,8)*+{B}="2";
(0,-8)*+{B^{\prime}}="4";
(-3,0)*+{}="5";
(16,0)*+{C}="6";
{\ar^{x} "0";"2"};
{\ar^{y} "2";"6"};
{\ar_{x^{\prime}} "0";"4"};
{\ar_{y^{\prime}} "4";"6"};
{\ar^{b}_{\cong} "2";"4"};
\endxy
\]	

A \emph{realization} of $\mathbb{E}$ is a correspondence $\mathfrak{s}$ that assigns to each $\mathbb{E}$-extension $\delta\in \mathbb{E}(C,A)$ an equivalence class
$$
\mathfrak{s}(\delta)=\xymatrix@C=0.8cm{[A\ar[r]^x & B\ar[r]^y & C]}
$$
of composable morphisms in $\mathcal{C}$, and which satisfies the compatibility conditions described in \cite[Definition~2.9]{NP}.
The triple $(\mathcal{C},\mathbb{E},\mathfrak{s})$ is called an \emph{extriangulated category} if the following conditions are satisfied:
\begin{itemize}
	\item[(1)] $\mathbb{E}\colon \mathcal{C}^{\rm op}\times \mathcal{C}\rightarrow \mathrm{Ab}$ is an additive bifunctor;
	\item[(2)] $\mathfrak{s}$ is an additive realization of $\mathbb{E}$;
	\item[(3)] the pair $(\mathbb{E},\mathfrak{s})$ satisfies the axioms (ET3), (ET3)$^{\rm op}$, (ET4), and (ET4)$^{\rm op}$ in \cite[Definition~2.12]{NP}.
\end{itemize}

In what follows, we often write $\C$ for an extriangulated category $(\C, \mathbb{E}, \mathfrak{s})$ for simplicity.
	
\begin{remark}
	Let $\delta\in \mathbb{E}(C,A)$ be an $\mathbb{E}$-extension. By the functoriality of $\mathbb{E}$, for any morphisms $a\in \mathcal{C}(A,A')$ and $c\in \mathcal{C}(C',C)$, we obtain induced $\mathbb{E}$-extensions
	$$
	\mathbb{E}(C,a)(\delta)\in \mathbb{E}(C,A') \quad \text{and} \quad \mathbb{E}(c,A)(\delta)\in \mathbb{E}(C',A).
	$$
	For simplicity, we denote these extensions by $a_{*}\delta$ and $c^{*}\delta$, respectively.
\end{remark}

\begin{definition}{\rm \cite[Definition 3.1]{NP}}\label{DefYoneda}
	Assume that $(\C,\E,\s)$ be an extriangulated category.
	By Yoneda's lemma, any $\mathbb{E}$-extension $\delta\in\mathbb{E}(C,A)$ induces natural transformations
	\[ \delta_\sharp\colon\C(-,C)\rightarrow\mathbb{E}(-,A)\ \ \text{and}\ \ \delta^\sharp\colon\C(A,-)\rightarrow\mathbb{E}(C,-). \]
	For any $X\in\C$, these $\delta_\sharp^{X}$ and $\delta^\sharp_X$ are given as follows.
	\begin{enumerate}
		\item[(1)] $\delta_\sharp^{X}\colon\C(X,C)\to\mathbb{E}(X,A)\ ;\ f\mapsto f^{\ast}\delta$.
		\item[(2)] $\delta^\sharp_X\colon\C(A,X)\to\mathbb{E}(C,X)\ ;\ g\mapsto g_{\ast}\delta$.
	\end{enumerate}	
\end{definition}

\begin{definition}{\rm \cite[Definition 2.15]{NP}}\label{DefTermExact1}
	Let $(\C,\mathbb{E},\mathfrak{s})$ be an extriangulated category.
	\begin{enumerate}
		\item[(1)] A sequence $A\overset{x}{\longrightarrow}B\overset{y}{\longrightarrow}C$ is called an {\it conflation} if it realizes some $\mathbb{E}$-extension $\delta\in\mathbb{E}(C,A)$.
		\item[(2)] A morphism $x\in\C(A,B)$ is called an {\it inflation} if it admits some conflation $A\overset{x}{\ra}B\ra C$.
		\item[(3)] A morphism $x\in\C(A,B)$ is called an {\it deflation} if it admits some conflation $K\ra A\overset{x}{\ra}B$.
	\end{enumerate}	
\end{definition}

\begin{definition} \cite[Definition 2.19]{NP}
	If a conflation $A\overset{x}{\longrightarrow}B\overset{y}{\longrightarrow}C$ realize $\delta \in \E(C,A)$, then we write it
	$$A\overset{x}{\longrightarrow}B\overset{y}{\longrightarrow}C\stackrel{\delta}{\dra}$$
	and call it an $\E$-triangle. If a triplet $(a,b,c)$ realizes $(a,c):\delta \ra \delta^{\prime}$, then we write it as
	$$\xymatrix{A \ar[r]^{x} \ar[d]^{a} &B \ar[r]^{y} \ar[d]^{b} &C \ar@{-->}[r]^{\delta} \ar[d]^{c} &\\
		A^{\prime} \ar[r]^{x^{\prime}} &B^{\prime} \ar[r]^{y^{\prime}}  &C^{\prime} \ar@{-->}[r]^{\delta^{\prime}} &}$$
	and call it a morphism of $\E$-triangles.
\end{definition}

It follows directly from the definition that we have the following decomposition for morphisms between $\mathbb{E}$-triangles.

\begin{remark} \label{factorization of E-extensions}
	The morphism of $\E$-triangles factors through an $\E$-triangle. Let $(a,b,c)$ be a morphism between $A \stackrel{x}{\lra} B \stackrel{y}{\lra} C \stackrel{\delta}{\dra}$ and $ A^{\prime} \stackrel{x^{\prime}}{\lra} B^{\prime} \stackrel{y^{\prime}}{\lra} C^{\prime} \stackrel{\delta^{\prime}}{\dra}$. Then there is the following commutative diagram
	$$\xymatrix{A \ar[r]^{x} \ar[d]^{a} &B \ar[r]^{y} \ar[d]^{b^{\prime}} &C \ar@{-->}[r]^{\delta} \ar@{=}[d] &\\
		A^{\prime} \ar[r]^{l} \ar@{=}[d] &D \ar[r]^{m} \ar[d]^{b^{\prime\prime}} &C \ar@{-->}[r]^{a_{*}\delta}_{c^{*}\delta^{\prime}} \ar[d]^{c} &\\
		A^{\prime} \ar[r]^{x^{\prime}} &B^{\prime} \ar[r]^{y^{\prime}} &C^{\prime} \ar@{-->}[r]^{\delta^{\prime}} &,}$$
	where $b=b^{\prime\prime}b^{\prime}$ and $a_{*}\delta=c^{*}\delta^{\prime}$.
\end{remark}

\begin{proposition}
	Let $i\colon A \ra B,\ j\colon X \ra Y$ be morphisms in $\C$. Then $\E(i,j)=\E(A,j)\E(i,X)=\E(i,Y)\E(B,j)$. That is to say, $j_{*}i^{*}\delta=i^{*}j_{*}\delta$ for any $\delta \in \E(B,X)$.	
\end{proposition}

\begin{proof}
	For any $\E$-extension $\delta$ realized by a conflation $X \ra E \ra B$, consider the following morphisms of $\E$-triangles
	$$\xymatrix{X \ar[r] \ar@{=}[d] &E^{\prime} \ar[r] \ar[d] &A \ar@{-->}[r]^{i^{*}\delta} \ar[d]^{i} &\\
		X \ar[r] \ar[d]^{j} &E \ar[r] \ar[d] &B \ar@{-->}[r]^{\delta} \ar@{=}[d] &\\
		Y \ar[r] &E^{\prime\prime} \ar[r] &B \ar@{-->}[r]^{j_{*}\delta} &.}$$
By Remark \ref{factorization of E-extensions}, we obtain the following commutative diagram such that $j_{*}i^{*}\delta=i^{*}j_{*}\delta$.
	$$\xymatrix@C=0.45cm@R=0.45cm{ &X \ar[rr] \ar@{=}[dd] \ar[ld]_{j} & &E^{\prime} \ar[rr] \ar[dd] \ar[ld] & &A \ar@{-->}[rr]^{i^{*}\delta} \ar[dd]^(.2){i} \ar@{=}[ld]& &\\
		Y \ar[rr] \ar@{=}[dd] & &D \ar[rr] \ar[dd] & &A \ar@{-->}[rr] \ar[dd]^(.2){i} &&\\
		&X \ar[rr] \ar[ld]_{j} & & E \ar[rr] \ar[ld] & & B \ar@{-->}[rr]^{\delta} \ar@{=}[ld] & &\\
		Y \ar[rr] & &E^{\prime\prime} \ar[rr] & &B \ar@{-->}[rr]^{j_{*}\delta} &&.}$$	
\end{proof}

\begin{proposition}{\rm \cite[Propositions 3.3]{NP}}\label{FactExact}
	Let $A\overset{x}{\longrightarrow}B\overset{y}{\longrightarrow}C\overset{\delta}{\dashrightarrow}$ be any $\E$-triangle.
	Then the following sequences of natural transformations are exact.
	\[ \C(C,-)\xrightarrow{\C(y,-)}\C(B,-)\xrightarrow{\C(x,-)}\C(A,-)\overset{\ \delta^{\sharp}\ }{\lra}\mathbb{E}(C,-)\xrightarrow{\mathbb{E}(y,-) }\mathbb{E}(B,-); \]
	\[ \C(-,A)\xrightarrow{\C(-,x)}\C(-,B)\xrightarrow{\C(-,y)}\C(-,C)\overset{\ \delta_{\sharp}\ }{\lra}\mathbb{E}(-,A)\xrightarrow{\mathbb{E}(-,x) }\mathbb{E}(-,B). \]	
\end{proposition}

\begin{proposition} \label{Ext=0 and factorization}
Let $(\C,\mathbb{E},\mathfrak{s})$ be an extriangulated category, $\I, \J$ be ideals of $\C$ and $i\colon A \ra B,\ j\colon X \ra Y$ be morphisms in $\I$ and $\J$ respectively. Consider the following diagrams of $\E$-triangles
	\begin{equation}
		\begin{array}{l} \label{E-extension diagram 1}
			$$\xymatrix{X \ar[r]^{f} \ar@{=}[d] &E \ar[r]^{g} &B \ar@{-->}[r]^{\delta} &\\
				X \ar[r]^{l} \ar[d]^{j} &F \ar[r]^{m}  \ar[u] \ar[d] &A \ar@{-->}[r]^{i^{*}\delta} \ar@{=}[d] \ar[u]^{i}&\\
				Y \ar[r] &D \ar[r] &A \ar@{-->}[r]^{j_{*}i^{*}\delta} &}$$
			
		\end{array}
	\end{equation}
and
	\begin{equation}
		\begin{array}{l} \label{E-extension diagram 2}		
			$$\xymatrix{X \ar[r]^{f} \ar[d]^{j} &E \ar[r]^{g} \ar[d] &B \ar@{-->}[r]^{\delta} \ar@{=}[d] &\\
				Y \ar[r]^{n} \ar@{=}[d] &G \ar[r]^{a} &B \ar@{-->}[r]^{j_{*}\delta} &\\
				Y \ar[r] &D^{\prime} \ar[r]  \ar[u] &A \ar@{-->}[r]^{i^{*}j_{*}\delta}  \ar[u]^{i}&.}$$
			
		\end{array}	
	\end{equation}	
we have that
     $\E(i,j)=0$ if and only if $j$ factors through $l$ in diagram {\rm (\ref{E-extension diagram 1})} if and only if $i$ factors through $a$ in diagram {\rm (\ref{E-extension diagram 2})}.
\end{proposition}
	
\begin{proof}
	Applying $\Hom_{\C}(-,Y)$ to the second row in the diagram (\ref{E-extension diagram 1}), we obtain an exact sequence
	$$\C(F,Y) \xrightarrow{\C(l,Y)} \C(X,Y) \xrightarrow{(i^{*}\delta)^{\sharp}_{Y}} \E(A,Y).$$	
	Then $j$ factors through $l$ if and only if $j \in \im\C(l,Y)=\ker(i^{*}\delta)^{\sharp}_{Y}$ if and only if $j_{*}i^{*}\delta=(i^{*}\delta)^{\sharp}_{Y}(j)=0$ if and only if $\E(i,j)=0$.
	
	Applying $\Hom_{\C}(A,-)$ to the second row in the diagram (\ref{E-extension diagram 2}), we obtain an exact sequence
	$$\C(A,G) \xrightarrow{\C(A,a)} \C(A,B) \xrightarrow{(j_{*}\delta)_{\sharp}^{A}} \E(A,Y).$$
	
Then $i$ factors through $a$ if and only if $i \in \im\C(A,a)=\ker(j_{*}\delta)_{\sharp}^{A}$ if and only if $i^{*}j_{*}\delta=(j_{*}\delta)_{\sharp}^{A}(i)=0$ if and only if $\E(i,j)=0$. 	
\end{proof}

\begin{proposition} {\rm \cite[Corollary 3.16]{NP}} \label{cor3.16}
	Let $(\C,\mathbb{E},\mathfrak{s})$ be an extriangulated category.
	\begin{itemize}
		\item [\rm (1)] Let $x \in \C(A,B),\ f\in \C(A,D)$ be any pair of morphisms. If $x$ is an inflation, then so is $\mv{f}{x} \in \C(A,D\oplus B)$.
		\item [\rm (2)] Let $y \in \C(B,C),\ f\in \C(D,C)$ be any pair of morphisms. If $y$ is a deflation, then so is $(y\ f) \in \C(B\oplus D,C)$.
	\end{itemize}	
\end{proposition}

\subsection{Higher extensions}

\begin{definition}{\rm \cite[Definition 3.23]{NP}}\label{projective-injective}
	Let $(\C,\mathbb{E},\mathfrak{s})$ be an extriangulated category.
	\begin{enumerate}
		\item [(1)] An object $P\in\C$ is called {\it projective} if it satisfies $\E(P,\C)=0$. We denote the subcategory of projective objects by $\mathcal P\subseteq\C$. We say that $\C$ {\it has enough projectives} if any object $C\in\C$ admits a deflation $P\to C$ from some $P\in\mathcal P$.
		\item [(2)] Dually, an object $I\in\C$ is called {\it injective} if it satisfies $\E(\C,I)=0$. We denote the subcategory of injective objects by $\mathcal I\subseteq\C$. We say that $\C$ {\it has enough injectives} if any object $C\in\C$ admits a inflation $C\to I$ to some $I\in\mathcal I$.
	\end{enumerate}	
\end{definition}

\begin{definition} \cite[Definiton 4.2]{NP}
	Let $(\mathcal{C},\mathbb{E},\mathfrak{s})$ be an extriangulated category, and let $\mathcal{X}$ and $\mathcal{Y}$ be subcategories of $\mathcal{C}$.  Define $\Cone(\X,\Y)$ and $\CoCone(\X,\Y)$ of $\C$ as follows.
	\begin{itemize}
		\item[(1)] An object $C\in\mathcal{C}$ belongs to $\Cone(\X,\Y)$ if and only if there exists an $\mathbb{E}$-triangle
		$$X \longrightarrow Y \longrightarrow C \dashrightarrow$$
		with $X\in\X$ and $Y\in\Y$.
		\item[(2)] An object $C\in\mathcal{C}$ belongs to $\CoCone(\X,\Y)$ if and only if there exists an $\mathbb{E}$-triangle
		$$C \longrightarrow X \longrightarrow Y \dashrightarrow$$
		with $X\in\X$ and $Y\in\Y$.
	\end{itemize}
\end{definition}

Liu and Nakaoka \cite{LN} define higher extensions as follows. $$\Omega\X=\CoCone(\P,\X),\ \Sigma\X=\Cone(\X,\I),$$
are called the {\it syzygy} and {\it cosyzygy} of $\X$ respectively. Put $\Omega^{0}\X=\Sigma^{0}\X=\X$. For any integer $k> 0$,  define the {\it k-th syzygy} and {\it k-th cosyzygy} of $\X$ by
$$\Omega^{k}\X=\Omega(\Omega^{k-1}\X)=\CoCone(\P,\Omega^{k-1}\X),$$
$$\Sigma^{k}\X=\Sigma(\Sigma^{k-1}\X)=\Cone(\Sigma^{k-1}\X,\I).$$

For objects $A,B \in \mathcal{C}$, \cite[Lemma 5.1]{LN} shows that $\E(A,\Sigma^{k}B)\cong \E(\Omega^{k}A,B)$. For convenience, we denote $\E(A,\Sigma^{k}B)\cong \E(\Omega^{k}A,B)$ by $\E^{k+1}(A,B)$ for any integer $k \ge 0$.

Similar to the higher extensions of objects, we define the higher extensions of morphisms.	
For any morphism $j\colon X\ra Y$, define $\Sigma j\colon \Sigma X\ra \Sigma Y$ to be any morphism such that the following diagram is commutative
$$\xymatrix{X
	\ar[r] \ar[d]^{j} &I(X) \ar[r] \ar[d] &\Sigma X \ar@{-->}[r] \ar[d]^{\Sigma j} &\\
	Y \ar[r] &I(Y) \ar[r] &\Sigma Y \ar@{-->}[r] &.}$$

Note that
$\Sigma X,\ \Sigma j$ are not unique. However, $\E(A,\Sigma X),\ \E(A,\Sigma j)$
are uniquely determined up to isomorphism. On the one hand, take two 1-th cosygyzys of $X$, denoted by $\Sigma_{1} X$ and $\Sigma_{2} X$ respectively. Then there exists a commutative diagram
$$\xymatrix{X \ar[r] \ar@{=}[d]^{1_{X}} &I_{1}(X) \ar[r] \ar@{-->}[d] &\Sigma_{1} X \ar@{-->}[d]^{f} \ar@{-->}[r] &\\
	X \ar[r]  & I_{2}(X) \ar[r] &\Sigma_{2} X \ar@{-->}[r] &.}$$
Take a 1-th sygyzy of $A$, thus we have the following commutative diagram by \cite[Lemma 5.1]{LN}
$$\xymatrix{\E(\Omega A,X) \ar[r]^{\varphi_{A,X_{1}}} \ar@{=}[d] &\E(A,\Sigma_{1}X) \ar[d]^{\E(A,f)}_{\cong}\\
	\E(\Omega A,X) \ar[r]^{\varphi_{A,X_{2}}}  &\E(A,\Sigma_{2}X),} $$
where $\varphi$ is isomorphic to its inverse, denoted by $\psi$.
Therefore, $\E(A,\Sigma_{1} X)\cong\E(A,\Sigma_{2} X)$.

On the other hand, assume that both $g^{\prime}$ and $g^{\prime\prime}$ make the following diagram commutative
$$\xymatrix{X
	\ar[r] \ar[d]^{j} &I(X) \ar[r] \ar@{-->}[d] &\Sigma X \ar@{-->}[r] \ar@{-->}[d]^{g^{\prime\prime}} _{g^{\prime}}&\\
	Y \ar[r] &I(Y) \ar[r] &\Sigma Y \ar@{-->}[r] &.}$$
Then both $\E(A,g^{\prime})$ and $\E(A,g^{\prime\prime})$ make the following diagram commutative
$$\xymatrix{\E(\Omega A,X) \ar[r]^{\varphi_{A,X}} \ar[d]_{\E(\Omega A,j)} &\E(A,\Sigma X) \ar[d]^{\E(A,g^{\prime\prime})}_{\E(A,g^{\prime})}\\
	\E(\Omega A,Y) \ar[r]^{\varphi_{A,Y}}  &\E(A,\Sigma Y).} $$
Therefore, $\E(A,g^{\prime})\cong\E(A,g^{\prime\prime})$. In this case, we denote both
$g^{\prime}$ and $g^{\prime\prime}$
uniformly by $\Sigma j$.

Furthermore, for any morphisms $i\colon A\ra B, j\colon X \ra Y$, we obtain four $\E$-triangles by taking 1-th cosygyzies of $X$ and $Y$, 1-th sygyzies of $A$ and $B$:
\begin{equation} \label{E-triangle(1)}
	X \lra I(X) \lra \Sigma X \stackrel{\delta_{X}}{\dra},
\end{equation}
\begin{equation} \label{E-triangle(2)}
	Y \lra I(Y) \lra \Sigma Y \stackrel{\delta_{Y}}{\dra},
\end{equation}
\begin{equation} \label{E-triangle(3)}
	\Omega A \lra P(A) \lra A \stackrel{\lambda_{A}}{\dra},
\end{equation}
\begin{equation} \label{E-triangle(4)}
	\Omega B \lra P(B) \lra B \stackrel{\lambda_{B}}{\dra}.
\end{equation}	
Then there exists a commutative diagram
$$\xymatrix{\E(B,\Sigma X) \ar @{} [dr] |{(\rm\uppercase\expandafter{\romannumeral 1})} \ar[r]^{\E(i,\Sigma X)} \ar[d]_{\cong}^{\psi_{B,X}} &\E(A,\Sigma X) \ar @{} [dr] |{(\rm\uppercase\expandafter{\romannumeral 2})} \ar[r]^{\E(A,\Sigma j)} \ar[d]_{\cong}^{\psi_{A,X}} &\E(A,\Sigma Y) \ar[d]_{\cong}^{\psi_{A,Y}}\\
	\E(\Omega B,X) \ar[r]_{\E(\Omega i,X)} &\E(\Omega A,X) \ar[r]_{\E(\Omega A,j)}  &\E(\Omega A,Y)}$$
which shows that $\E(i,\Sigma X)\cong \E(\Omega i,X)$ and $\E(A,\Sigma j)\cong \E(\Omega A,j)$.	

Thus, we can define a covariant functor
$$\begin{aligned}
	\E^{2}(A,-)\colon \C &\rightarrow \Ab\\
	X &\mapsto \E^{2}(A,X)\colon=\E(A,\Sigma X)\cong\E(\Omega A,X)\\
	j &\mapsto \E^{2}(A,j)\colon=\E(A,\Sigma j)\cong\E(\Omega A,j)
\end{aligned}$$
and a contravariant functor
$$\begin{aligned}
	\E^{2}(-,X)\colon \C &\rightarrow \Ab\\
	A &\mapsto \E^{2}(A,X)\colon=\E(A,\Sigma X)\cong\E(\Omega A,X)\\
	i &\mapsto \E^{2}(i,X)\colon=\E(i,\Sigma X)\cong\E(\Omega i,X).
\end{aligned}$$

Similarly, we define higher functors. For any positive integer $k$,
$$\begin{aligned}
	\E^{k}(A,-)\colon \C &\rightarrow \Ab\\
	X &\mapsto \E^{k}(A,X)\colon=\E(A,\Sigma^{k-1} X)\cong\E(\Omega^{k-1} A,X)\\
	j &\mapsto \E^{k}(A,j)\colon=\E(A,\Sigma^{k-1} j)\cong\E(\Omega^{k-1} A,j)
\end{aligned}$$
and
$$\begin{aligned}
	\E^{k}(-,X)\colon \C &\rightarrow \Ab\\
	A &\mapsto \E^{k}(A,X)\colon=\E(A,\Sigma^{k-1} X)\cong\E(\Omega^{k-1} A,X)\\
	i &\mapsto \E^{k}(i,X)\colon=\E(i,\Sigma^{k-1} X)\cong\E(\Omega^{k-1} i,X).
\end{aligned}$$	
For convenience, we denote $\E(i,\Sigma^{k-1} j)\cong \E(\Omega^{k-1} i,j)$ by $\E^{k}(i,j)$ for any positive integer $k$.

\subsection{Orthogonality of ideals and ideal approximations}
Let $(\C,\E,\s)$ be an extriangulated category, $\I, \J$ be ideals of $\C$, $k$ any positive integer. An ideal $\I$ in $\C$ is an additive two-sided ideal of its morphism category, i.e., a collection of subgroups $\I(X,Y)\subseteq\Hom_{\C}(X,Y)$ closed under composition with arbitrary morphisms on both sides for any objects $X,\ Y$ in $\C$.

Denote ${\rm Ob}(\I)=\{X \in \C \ |\ 1_{X} \in \I \}$. It is easy to see that ${\rm Ob}(\I)$ is closed under direct sums and direct summands.

Two morphisms $f\colon X \ra Y,\ g\colon A \ra B$ are called $\Hom$-orthogonal if $\Hom_{\C}(f,g)=0$.
Let $\S$ be a collection of morphisms of $\C$, define left $\Hom$-orthogonal ideal and right $\Hom$-orthogonal ideal as follows:
$$^{\bot} \S=\{f \ | \ \Hom_{\C}(f,s)=0 \ \text{for any morphism}\ s \in \S \},$$
$$\S^{\bot}=\{g \ | \ \Hom_{\C}(s,g)=0 \ \text{for any morphism}\ s \in \S \}.$$

Two morphisms $f\colon X \ra Y,\ g\colon A \ra B$ are called Ext-orthogonal if $\E(f,g)=0$.
Let $\S$ be a collection of morphisms of $\C$, define left $k$-th Ext-orthogonal ideal and right $k$-th Ext-orthogonal ideal as follows:
$$^{\bot_{k}} \S=\{f \ | \ \E^{k}(f,s)=0 \ \text{for any morphism}\ s \in \S \},$$
$$\S^{\bot_{k}}=\{g \ | \ \E^{k}(s,g)=0 \ \text{for any morphism}\ s \in \S \}.$$

%\begin{lemma}
%	For any objects $A,\ B\in \C$, $\E(1_{A},1_{B})=0$ if and only if $\E(A,B)=0$.
%\end{lemma}

%\begin{lemma} \label{Ext,Hom,composition}
%	${\rm Hom}_{\C}(\I,\J)=0$ if and only if $\mathcal{J}\mathcal{I}=0$. 	
%\end{lemma}

%\begin{proof}
%	Assume that ${\rm Hom}_{\C}(\I,\J)=0$ and let $i:X \rightarrow Y \in \mathcal{I},\ j:Y \rightarrow Z \in \mathcal{J}$ be morphisms. Consider the following composition
%	$$\C(Y,Y) \xrightarrow{\C(i,Y)} \C(X,Y) \xrightarrow{\C(X,j)}  \C(X,Z),$$
%	$ji=\C(X,j)\circ\C(i,Y)(1_{Y})=\C(i,j)(1_{Y})=0$.
%	Conversely, assume that $\mathcal{J} \mathcal{I}$=0 and let $i:X \rightarrow Y \in \mathcal{I},\ j:A \rightarrow B \in \mathcal{J}$ be morphisms.
%	Consider the following composition
%	$$\C(Y,A) \xrightarrow{\C(i,A)} \C(X,A) \xrightarrow{\C(X,j)}  \C(X,B),$$
%	for any morphism $f:Y \rightarrow A$, $\C(i,j)(f)=jfi=0$ since $fi \in \mathcal{I}$.	
%\end{proof}

A morphism $f\colon X \ra M$ in $\I$ is called an $\I$-precover, if for any morphism $g\colon X^{\prime} \ra M$ in $\I$, there exists a morphism $h\colon X^{\prime} \ra X$ such that $g=fh$. A morphism $f\colon M \ra Y$ in $\J$ is called a $\J$-preenvelope, if for any morphism $g\colon M \ra Y^{\prime}$ in $\J$, there exists a morphism $h\colon Y \ra Y^{\prime}$ such that $g=hf$.

The ideal $\I$ is precovering, if for any object $M \in \C$, there exists an $\I$-precover of $M$. The ideal $\J$ is preenveloping, if for any object $M \in \C$, there exists a $\J$-preenvelope of $M$.

A morphism $f\colon X \ra M$ in $\I$ is called a special $\I$-precover, if $f$ is a deflation and $f$ can be embedded the diagram between $\E$-triangles
$$\xymatrix{A \ar[r] \ar[d]_{h} &B \ar[r] \ar[d] &M \ar@{=}[d] \ar@{-->}[r]^{\lambda} &\\
	K \ar[r]  &X \ar[r]^{f} &M \ar@{-->}[r]^{h_{*}\lambda} &,}$$
where $h \in \I^{\bot_{1}}$.

A morphism $g\colon M \ra Y$ in $\J$ is called a special $\J$-preenvelope, if $g$ is an inflation and $g$ can be embedded the diagram between $\E$-triangles
$$\xymatrix{M \ar[r]^{g} \ar@{=}[d]  &Y \ar[r] \ar[d] &C \ar[d]^{h} \ar@{-->}[r]^{h^{*}\delta} &\\	M \ar[r]  &A \ar[r] &B \ar@{-->}[r]^{\delta} &,}$$
where $h \in {^{\bot_{1}}\J}$.

An ideal $\mathcal{I}$ of $\mathcal{C}$ is called special precovering if every object $M \in \mathcal{C}$ admits a special $\mathcal{I}$-precover;
dually, an ideal $\mathcal{J}$ is called special preenveloping if every object $M \in \mathcal{C}$ admits a special $\mathcal{J}$-preenvelope.

\subsection{Triangulated structures of stable categories}

Let $\C$ be a Frobenius extriangulated category, $\omega$ the subcategory of projective-injective objects. The objects of quotient category $\C/\omega$ are the same as those of $\C$ and the morphisms are equivalence classes of morphisms factoring through objects in $\omega$, i.e.,
$$\Hom_{\C/\omega}(X,Y)=\Hom_{\C}(X,Y)/\omega(X,Y)$$
where $\omega(X,Y)=\{f \in \Hom_{\C}(X,Y)\ | \ f \ \text{factors through some object in}\ \omega\}$. Denote the quotient category $\C/\omega$ by $\underline{\C}$.

For any $\E$-triangle
$$X \lra Y \lra Z \dra,$$ consider the following commutative diagram
$$\xymatrix{X \ar[r]^{f} \ar@{=}[d]& Y \ar[r]^{g} \ar[d]& Z \ar@{-->}[r] \ar[d]^{h}& \\
	X \ar[r] & I(X) \ar[r] & \Sigma X \ar@{-->}[r] &,}$$
where $I(X) \in \omega$ and denote the shift of $X$ in $\underline{\C}$ by $\Sigma X$. The sequence
$$X \stackrel{\underline{f}}{\lra}Y \stackrel{\underline{g}}{\lra} Z \stackrel{\underline{h}}{\lra} \Sigma X$$
is the standard triangle in $\underline{\C}$.	
Then $\underline{\mathcal{C}}:=\mathcal{C}/\omega$ is a triangulated category with the standard triangles defined above; see \cite[Corollary 7.4] {NP} and \cite[Theorem 3.13]{ZZ}.
In particular,
we have that $\Sigma\colon\underline{\C}	\ra  \underline{\C}$ is an automorphism.
For any object $X,\ \Sigma X$ is unique in $\underline{\C}$. For any morphism $\underline{x}$, $\Sigma\underline{x}\triangleq\underline{\Sigma x}$ is unique in $\underline{\C}$.

For any ideal $\J$ in $\C$, we have
$$\Sigma (\J/\omega)= (\Sigma \J)/\omega,~\Sigma^{2} (\J/\omega)=\Sigma(\Sigma (\J/\omega))=\Sigma((\Sigma \J)/\omega)=(\Sigma^{2} \J)/\omega,~\Sigma^{k} (\J/\omega)=(\Sigma^{k} \J)/\omega$$ for any positive integer $k$.
In what follows, we abbreviate $(\Sigma ^{k}\J)/\omega$ as $\Sigma^{k} \J/\omega$.

%$$ \Sigma:\underline{\C} \xrightarrow{\ \ \ \ \ } \underline{\C}\ \ \ \ \ \ \ \ \ \ \ \ \ $$
%$$\xymatrix{
%X \ar[r] \ar[d]^{\underline{x}} &\Sigma X \ar[d]^{\Sigma\underline{x}\triangleq\underline{\Sigma x}}\\
%A \ar[r] &\Sigma A}$$

For any morphism between $\E$-triangles
$$\xymatrix{X \ar[r]^{f} \ar[d]^{x}& Y \ar[r]^{g} \ar[d]^{y}& Z \ar@{-->}[r] \ar[d]^{z}& \\
	A \ar[r]^{a} & B \ar[r]^{b} & C \ar@{-->}[r] &,}$$	
we obtain a commutative diagram of triangles in $\underline{\C}$
$$\xymatrix{X \ar[r]^{\underline{f}} \ar[d]^{\underline{x}}& Y \ar[r]^{\underline{g}} \ar[d]^{\underline{y}}& Z \ar[r]^{\underline{h}} \ar[d]^{\underline{z}}&\Sigma X \ar[d]^{\Sigma\underline{x}\triangleq\underline{\Sigma x}}\\
	A \ar[r]^{\underline{a}} & B \ar[r]^{\underline{b}} & C \ar[r]^{\underline{c}} &\Sigma A.}$$

Note that in the stable triangulated category $\underline{\C}$, the additive bifunctor $\E(-,-)=\underline{\C}(-,\Sigma -)$. For any ideal $\I/\omega,\J/\omega$ of $\underline{\C}$, $\E^{k}(\I/\omega,\J/\omega)=\underline{\C}(\I/\omega,\Sigma^{k} \J/\omega)$ for any positive integer $k$.

\section{Comparison between complete cotorsion pairs and complete ideal cotorsion pairs}

\setcounter{equation}{0}
	
We begin by recalling the definitions of complete cotorsion pairs and complete ideal cotorsion pairs.

\begin{definition}\cite[Definition 4.1]{NP}
Let $(\C,\E,\s)$ be an extriangulated category, $\X, \Y$ subcategories closed under isomorphisms and direct summands. The pair $(\X,\Y)$ is called a complete cotorsion pair, if it satisfies the following conditions.
\begin{itemize}
	\item [(1)] $\E(\X,\Y)=0$.
	\item [(2)] For any $C\in \C$, there exists an $\E$-triangle
    $$Y^{C} \lra X^{C} \lra C \dra,$$
    where $X^{C}\in \X, Y^{C} \in \Y$.
	\item [(3)] For any $C\in \C$, there exists an $\E$-triangle
	$$C \lra Y_{C} \lra X_{C} \dra,$$
	where $X_{C}\in \X, Y_{C} \in \Y$.
\end{itemize}
\end{definition}
%
%Let $(\C,\E,\s)$ be an extriangulated category and $\I$ be an ideal of $\C$.
%A morphism $f\colon X \ra M$ in $\I$ is called a special $\I$-precover, if $f$ is a deflation and $f$ can be embedded the diagram between $\E$-triangles
%$$\xymatrix{A \ar[r] \ar[d]_{h} &B \ar[r] \ar[d] &M \ar@{=}[d] \ar@{-->}[r]^{\lambda} &\\
%	K \ar[r]  &X \ar[r]^{f} &M \ar@{-->}[r]^{h_{*}\lambda} &,}$$
%where $h \in \I^{\bot}$.
%
%A morphism $g\colon M \ra Y$ in $\J$ is called a special $\J$-preenvelope, if $g$ is an inflation and $g$ can be embedded the diagram between $\E$-triangles
%$$\xymatrix{M \ar[r]^{g} \ar@{=}[d]  &Y \ar[r] \ar[d] &C \ar[d]^{h} \ar@{-->}[r]^{h^{*}\delta} &\\	M \ar[r]  &A \ar[r] &B \ar@{-->}[r]^{\delta} &,}$$
%where $h \in {^{\bot}\J}$.
%
%An ideal $\mathcal{I}$ of $\mathcal{C}$ is called \textbf{special precovering} if every object $M \in \mathcal{C}$ admits a special $\mathcal{I}$-precover;
%dually, an ideal $\mathcal{J}$ is called \textbf{special preenveloping} if every object $M \in \mathcal{C}$ admits a special $\mathcal{J}$-preenvelope.

\begin{definition} \cite[Definition 3.8]{ZH}
	Let $(\C,\E,\s)$ be an extriangulated category. An ideal pair $(\mathcal{I},\mathcal{J})$ is called an ideal cotorsion pair, if $\mathcal{I}^{\bot_{1}}=\mathcal{J},\ \mathcal{I}={^{\bot_{1}}\mathcal{J}}$. Moreover, an ideal cotorsion pair $(\I,\J)$ is complete, if $\I$ is a special precovering ideal and $\J$ is a special preenveloping ideal.	
\end{definition}

\begin{remark}
		(1) \cite[Theorem 5.3.4]{BM}~In a triangulated category, $\I$ is special precovering if and only if $\I$ is precovering. When $(\I,\J)$ is an ideal cotorsion pair, $\I$ is special precovering if and only if $\J$ is special preenveloping.
	
	(2) \cite[Lemma 4.5]{ZH}~In an extriangulated category with enough projective and enough injective objects, when $(\I,\J)$ is an ideal cotorsion pair, $\I$ is special precovering if and only if $\J$ is special preenveloping.
\end{remark}

We collect some examples of complete ideal cotorsion pairs.

\begin{example}
	(1) \cite[Corollary 4.19]{H}
Let $R$ be a Gorenstein ring and let $\mathcal{P}^{<\infty}$ denote the class of right $R$-modules of finite projective dimension.
Recall that a morphism $f\colon A \to X$ is $\mathcal{G}$-phantom if the pullback of every short exact sequence with end term $X$ along $f$ is ${\rm Hom}_{R}(-,\mathcal{P}^{<\infty})$-exact.
A morphism $g\colon M \to N$ is $\mathcal{G}$-injective if the pushout of every ${\rm Hom}_{R}(-,\mathcal{P}^{<\infty})$-exact sequence with initial term $M$ along $g$ is split.
Let $\varphi(\mathcal{G}\text{-xt})$ denote the class of $\mathcal{G}$-phantom morphisms, and let $\mathcal{G}\text{-inj}$ denote the class of $\mathcal{G}$-injective morphisms.
Then $(\varphi(\mathcal{G}\text{-xt}),\mathcal{G}\text{-inj})$ is a complete ideal cotorsion pair in ${\rm Mod}R$.

	(2) \cite[Proposition 6.1.1]{BM}
	Let $\T$ be a triangulated category with the shift functor $\Sigma$. Recall a projective class in $\T$ is a pair $(\mathcal P,\J)$ where $\mathcal P$ is a class of objects and $\J$ is an ideal satisfying the following conditions:
	\begin{itemize}
		\item [(a)] $\mathcal P=\{P\in \T\ |\ \Hom_{\T}(P,j)=0\ \text{for all $j \in \J$} \}.$
		\item [(b)] $\J=\{j\in {\rm Mor}(\T)\ |\ \Hom_{\T}(P,j)=0\ \text{for all $P \in \mathcal P$} \}.$
		\item [(c)] For any object $T \in \T$, there exists a triangle $$P\lra T \stackrel{j}{\lra} Y \lra \Sigma P,$$ where $P\in \mathcal P, j \in \J$.
	\end{itemize}
If $(\mathcal P,\mathcal J)$ is a projective class such that both $\mathcal P$ and $\mathcal J$ are closed under the shift functor, then $([\mathcal P],\mathcal J)$ is a complete ideal cotorsion pair.
	
	(3) \cite[Corollary 3.15]{ZH} Let $\C$ be an extriangulated category, $\mathbb{F}$ an additive subfunctor of $\E$ with enough injective morphisms.
	Recall a morphism $\varphi\colon X \ra C$ is $\mathbb{F}$-phantom, if $\varphi^{*}\delta\in\mathbb{F}(X,A)$ for any $\delta \in \E(C,A)$.
	A morphism $f\colon A \ra B$ is $\mathbb{F}$-injective, if $f_{*}\delta=0$ for any $\delta \in \mathbb{F}(C,A)$.
	Denote $\mathbf{Ph}(\mathbb{F})$ the class of $\mathbb{F}$-phantom morphisms,  $\mathbb{F}\text{-}\mathbf{inj}$ the class of $\mathbb{F}$-injective morphisms. Then
	$(\mathbf{Ph}(\mathbb{F}),\mathbb{F}\text{-}\mathbf{inj})$ is a complete ideal cotorsion pair.
\end{example}

Let $(\mathcal{C},\mathbb{E},\mathfrak{s})$ be an extriangulated category, and denote by ${\rm Mor}(\mathcal{C})$ the class of morphisms in $\mathcal{C}$. Let $\mathcal{D}$ be a subcategory of $\mathcal{C}$. Define
$$
[\mathcal{D}] := \{ f \in {\rm Mor}(\mathcal{C}) \mid f \text{ factors through some object in } \mathcal{D} \}.
$$
It is straightforward to verify that $[\mathcal{D}]$ is an ideal of $\mathcal{C}$, called an \textbf{object ideal}. Moreover, an ideal $\mathcal{I}$ is an object ideal if and only if $\mathcal{I} = [{\rm Ob}(\mathcal{I})]$.

\begin{theorem}\label{main11}
Let $(\mathcal{C},\mathbb{E},\mathfrak{s})$ be an extriangulated category, and let $\mathcal{X}, \mathcal{Y}$ be subcategories of $\mathcal{C}$. If $(\mathcal{X},\mathcal{Y})$ is a complete cotorsion pair in $\mathcal{C}$, then $([\mathcal{X}],[\mathcal{Y}])$ is a complete ideal cotorsion pair in $\mathcal{C}$.
\end{theorem}

\begin{proof}
The inclusion $[\mathcal{X}] \subseteq {}^{\perp_1}[\mathcal{Y}]$ is immediate. For the reverse inclusion, let $f\colon A \to B$ be a morphism in ${}^{\perp_1}[\mathcal{Y}]$. Since $(\mathcal{X},\mathcal{Y})$ is a cotorsion pair, there exists an $\mathbb{E}$-triangle
$$Y^{B} \lra X^{B} \stackrel{g}{\lra} B\stackrel{\delta}{\dra},$$
with $X^B \in \mathcal{X}$ and $Y^B \in \mathcal{Y}$. Consider the following morphism of $\mathbb{E}$-triangles:
$$
\xymatrix{
Y^B \ar[r] \ar@{=}[d] & X' \ar[r] \ar[d] & A \ar[d]^{f} \ar@{-->}[r]^{f^*\delta} & \\
Y^B \ar[r] & X^B \ar[r]^{g} & B \ar@{-->}[r]^{\delta}&,
}
$$
since $f \in {}^{\perp_1}[\mathcal{Y}]$, we have $\mathbb{E}(f,[\mathcal{Y}])=0$, which implies that $f$ factors through $X^B$. Hence $f \in [\mathcal{X}]$, and therefore ${}^{\perp_1}[\mathcal{Y}] \subseteq [\mathcal{X}]$. Thus $[\mathcal{X}] = {}^{\perp_1}[\mathcal{Y}]$.

Moreover, the morphism $g\colon X^B \to B$ is a special $[\mathcal{X}]$-precover, as it fits into the commutative diagram of $\mathbb{E}$-triangles
$$
\xymatrix{
Y^B \ar[r] \ar@{=}[d] & X^B \ar[r] \ar@{=}[d] & B \ar@{=}[d] \ar@{-->}[r] & \\
Y^B \ar[r] & X^B \ar[r]^{g} & B \ar@{-->}[r]^{\delta} &,
}
$$
where $1_{Y^B} \in [\mathcal{X}]^{\perp_1}$. Hence $[\mathcal{X}]$ is special precovering.

The equality $[\mathcal{Y}] = [\mathcal{X}]^{\perp_1}$ and the fact that $[\mathcal{Y}]$ is special preenveloping follow dually.
\end{proof}

Since every exact category can be viewed as an extriangulated category, Theorem \ref{main11} yields the following consequence.

\begin{corollary}{\rm\cite[Theorem 28]{FAHT}}
Let $\mathcal{B}$ be an exact category, and let $\mathcal{X}, \mathcal{Y}$ be subcategories of $\mathcal{B}$. If $(\mathcal{X},\mathcal{Y})$ is a complete cotorsion pair in $\mathcal{B}$, then $([\mathcal{X}],[\mathcal{Y}])$ is a complete ideal cotorsion pair in $\mathcal{B}$.
\end{corollary}

We have just established that every complete cotorsion pair gives rise to a complete ideal cotorsion pair via the ideal generation construction $(\mathcal{X},\mathcal{Y}) \mapsto ([\mathcal{X}],[\mathcal{Y}])$. A natural question arises: does the converse hold? That is, given a complete ideal cotorsion pair $([\mathcal{X}],[\mathcal{Y}])$ generated by subcategories $\mathcal{X}$ and $\mathcal{Y}$, can we conclude that $(\mathcal{X},\mathcal{Y})$ is a complete cotorsion pair?
To prove this, we need the following some preparations.

We recall some notions. Let $\mathcal A$ be an additive category. A morphism $e \colon A \to A$ is called idempotent if $e^2 = e$. The category $\mathcal A$ is said to be idempotent complete (or to have split idempotents) if every idempotent morphism in $\mathcal A$ has a kernel.

An additive category $\mathcal A$ is weakly idempotent complete if every retraction has a kernel.

\begin{proposition}\label{prop11}
Let $\mathcal{A}$ be an additive category. If $\mathcal{A}$ is idempotent complete,
then $\mathcal{A}$ is weakly idempotent complete.
\end{proposition}

\begin{proof}
Let $r: A \to B$ be a retraction. Then there exists $s: B \to A$ such that
$r s = \mathrm{id}_B$.
Set $p := s r: A \to A$. Then we have
$$
p^2 = (sr)(sr) = s (rs)r = s\mathrm{id}_Br = sr = p.
$$
Thus $p$ is an idempotent morphism.
Since $\mathcal{A}$ is idempotent complete, $p$ has a kernel. Let
$k: K \to A$
be a kernel of $p$. Then $p \circ k = 0$.

We claim that $k: K \to A$ is also a kernel of $r$.

We first show that $rk = 0$.
Since $pk = 0$ and $p = sr$, we have $srk = 0$. Then
$$
rk = \mathrm{id}_B(rk) = (rs)(rk) = r(s rk) = r0 = 0.
$$
Now let $f: X \to A$ be any morphism such that $r \circ f = 0$.
Then
$$
p  f = s rf = s0 = 0.
$$
Since $k: K \to A$ is a kernel of $p$, the universal property of kernels gives a unique morphism $g: X \to K$ such that
$f = k \circ g$.
This shows that $k: K \to A$ is a kernel of $r$. Hence $\mathcal{A}$ is weakly idempotent complete.
\end{proof}

Recall that an additive category is called Krull-Schmidt if every object decomposes as a finite direct sum of indecomposable objects with local endomorphism rings, and this decomposition is unique up to isomorphism and permutation of summands.

By \cite[Corollary 4.4]{Krause} (see also \cite[Theorem A.1]{CYZ}), we know that an additive category $\mathcal{A}$ is Krull-Schmidt if and only if it is idempotent complete and ${\rm End}_{\mathcal{A}}(X)$ is semiperfect for every object $X \in \mathcal{A}$. Consequently, by Proposition \ref{prop11},  every Krull-Schmidt additive category is weakly idempotent complete.

\begin{definition}\cite[Condition 5.8]{NP}
Let $(\mathcal{C}, \mathbb{E}, \mathfrak{s})$ be an extriangulated category. We say that $\mathcal{C}$ satisfies \textbf{condition (WIC)} if the following statements hold:
\begin{enumerate}
\item[(1)] For any composable morphisms $f \in \mathcal{C}(X,Y)$ and $g \in \mathcal{C}(Y,Z)$, if $gf$ is an inflation, then $f$ is an inflation.
\item[(2)] For any composable morphisms $f \in \mathcal{C}(X,Y)$ and $g \in \mathcal{C}(Y,Z)$, if $gf$ is an deflation, then $g$ is an deflation.
\end{enumerate}
\end{definition}

\begin{lemma}{\rm \cite[Proposition 3.7]{K}}\label{lemwic}
 An extriangulated category satisfies {\bf condition
(WIC)} if and only if it is weakly idempotent complete.
\end{lemma}

\begin{lemma}{\rm\cite[Corollary 1.4 and its dual]{KS}}\label{right}
Let $\mathcal{A}$ be a Krull-Schmidt additive category. Then
\begin{enumerate}
\item[\rm (1)] For any morphism $f: X \to Y$ in $\mathcal{A}$, there exists a decomposition $X = X_1 \oplus X_2$, unique up to isomorphism, such that
    $f=(f_1,0)$ where $f_1: X_1 \to Y$ is right minimal;
\item[\rm (2)] For any morphism $f: X \to Y$ in $\mathcal{A}$, there exists a decomposition $Y = Y_1 \oplus Y_2$, unique up to isomorphism, such that
    $f=\binom{g_1}{0}$ where $g_1: X \to Y_1$ is left minimal.
\end{enumerate}
\end{lemma}

We now turn to the converse of Theorem~\ref{main11}. In full generality, this converse remains open. The difficulty lies in the fact that, for an arbitrary extriangulated category, a special $[\mathcal{X}]$-precover need not restrict to a special $\mathcal{X}$-precover on the object level, since the domain of the precover may not decompose into an $\mathcal{X}$-part and a $[\mathcal{X}]^{\perp_1}$-part.
However, if $\mathcal{C}$ is Krull-Schmidt, every object admits a unique decomposition into indecomposable summands, and every morphism admits a right minimal decomposition. This allows us to extract an $\mathcal{X}$-precover from a given $[\mathcal{X}]$-precover. We thus obtain the following positive result.

\begin{theorem}\label{main12}
Let $(\mathcal{C},\mathbb{E},\mathfrak{s})$ be a Krull--Schmidt extriangulated category, and let $\mathcal{X}, \mathcal{Y}$ be subcategories of $\mathcal{C}$. If $([\mathcal{X}],[\mathcal{Y}])$ is a complete ideal cotorsion pair in $\mathcal{C}$, then $(\mathcal{X},\mathcal{Y})$ is a complete cotorsion pair in $\mathcal{C}$.
\end{theorem}

\begin{proof}
(1) The condition $\mathbb{E}([\mathcal{X}],[\mathcal{Y}])=0$ immediately implies $\mathbb{E}(\mathcal{X},\mathcal{Y})=0$.

Before proving (2) and (3), we establish the equalities
$$
\mathcal{X} = {}^{\perp_1}\mathcal{Y}~~\text{and}~~\mathcal{Y} = \mathcal{X}^{\perp_1}.
$$
The inclusion $\mathcal{X} \subseteq {}^{\perp_1}\mathcal{Y}$ is immediate. For the reverse inclusion, let $M \in {}^{\perp_1}\mathcal{Y}$. Then $\mathbb{E}(M,\mathcal{Y}) = 0$, which gives $\mathbb{E}(1_M, [\mathcal{Y}]) = 0$. Hence $1_M \in {}^{\perp_1}[\mathcal{Y}] = [\mathcal{X}]$, so $M \in \mathcal{X}$. Thus $\mathcal{X} = {}^{\perp_1}\mathcal{Y}$. The dual argument yields $\mathcal{Y} = \mathcal{X}^{\perp_1}$.

Since $\mathcal{X} = {}^{\perp_1}\mathcal{Y}$ and $\mathcal{Y} = \mathcal{X}^{\perp_1}$, both $\mathcal{X}$ and $\mathcal{Y}$ are automatically closed under direct summands. Indeed, if $X = X_1 \oplus X_2 \in \mathcal{X}$, then for any $Y \in \mathcal{Y}$,
$$
\mathbb{E}(X_1, Y) \oplus \mathbb{E}(X_2, Y) \cong \mathbb{E}(X, Y) = 0,
$$
so $X_1 \in {}^{\perp_1}\mathcal{Y} = \mathcal{X}$. The proof for $\mathcal{Y}$ is dual.

(2) Let $M \in \mathcal{C}$. Since $[\mathcal{X}]$ is special precovering, there exists an $\mathbb{E}$-triangle
$$A\lra B \stackrel{f}{\lra} M \dra,$$
where $f: B \to M$ is a special $[\mathcal{X}]$-precover. As $f \in [\mathcal{X}]$, it factors through some object $X \in \mathcal{X}$, say $f = f_1f_2$ with $f_1: X \to M$ and $f_2: B \to X$. Then $f_1$ is an $\mathcal{X}$-precover of $M$.

Since $\mathcal{C}$ is Krull-Schmidt, the morphism $f_1$ admits a right minimal decomposition
$$
f_1 = (f_1', 0): X = X_1 \oplus X_2 \longrightarrow M,
$$
where $f_1': X_1 \to M$ is right minimal (see Lemma \ref{right}). Moreover, $f_1'$ is a deflation because $f_1 = (f_1',0) = f_1' (1,0)$ (see Lemma \ref{lemwic}). Hence there exists an $\mathbb{E}$-triangle
$$K \lra X_{1} \stackrel{f_{1}^{\prime}}{\lra} M \dra.$$
We claim that $f_1'$ is still an $\mathcal{X}$-precover. Indeed, for any morphism $g: X' \to M$, since $f_1$ is an $\mathcal{X}$-precover, there exists $h = \binom{h_1}{h_2}: X' \to X_1 \oplus X_2$ such that $f_1 h = g$, i.e.,
$$
(f_1', 0) \binom{h_1}{h_2} = f_1' h_1 = g.
$$
Thus $g$ factors through $f_1'$, so $f_1'$ is an $\mathcal{X}$-precover of $M$. Therefore $f_1'$ is a right minimal $\mathcal{X}$-precover. By \cite[Lemma 3.3]{CZZ}, we have $K \in \mathcal{X}^{\perp_1} = \mathcal{Y}$.

Since $X = X_1 \oplus X_2 \in \mathcal{X}$ and $\mathcal{X}$ is closed under direct summands, we have $X_1 \in \mathcal{X}$. Therefore, the $\mathbb{E}$-triangle
$$K \lra X_{1} \stackrel{f_{1}^{\prime}}{\lra} M \dra$$
is the desired one, with $X_1 \in \mathcal{X}$ and $K \in \mathcal{Y}$.

(3) This follows dually to (2).

Hence $(\mathcal{X},\mathcal{Y})$ is a complete cotorsion pair in $\C$.
\end{proof}

Since every exact category is extriangulated, Theorem~\ref{main12} applies in particular to exact categories. We thus obtain the following result.

\begin{corollary}\label{cor11}
Let $\mathcal{B}$ be a Krull-Schmidt exact category, and let $\mathcal{X}, \mathcal{Y}$ be subcategories of $\mathcal{B}$. If $([\mathcal{X}],[\mathcal{Y}])$ is a complete ideal cotorsion pair in $\mathcal{B}$, then $(\mathcal{X},\mathcal{Y})$ is a complete cotorsion pair in $\mathcal{B}$.
\end{corollary}

\begin{remark}
Corollary~\ref{cor11} provides a positive answer to the question raised in \cite[Question 28]{FAHT} in the setting of Krull--Schmidt exact categories. A related result was obtained in \cite[Corollary 1.3]{STWZ} for Krull--Schmidt Frobenius exact categories. However, our approach is entirely different from that of \cite{STWZ}: while \cite{STWZ} relies on stable category techniques specific to the Frobenius setting, our proof is direct and works in the broader context of arbitrary Krull--Schmidt exact categories.
\end{remark}

\section{Ideal $n$-cotorsion pairs in Frobenius extriangulated categories}

\setcounter{equation}{0}

Motivated by the corresponding notions of cotorsion pairs in triangulated categories, exact categories and extriangulated categories, we introduce the following definition.

\begin{definition}
Let $(\C,\E,\s)$ be an extriangulated category.
An ideal pair $(\I,\J)$ is called an ideal $n$-cotorsion pair if the following conditions are satisfied:
\begin{itemize}
\item [(1)] $\I=\bigcap\limits_{k=1}^n {^{\bot_{k}}\J}$;
\item [(2)] $\J=\bigcap\limits_{k=1}^n \I^{\bot_{k}}$;
\item [(3)] $\I$ is precovering and $\J$ is preenveloping.
\end{itemize}
\end{definition}

\begin{remark}
It follows directly from the definitions that an ideal $1$-cotorsion pair is precisely a complete ideal cotorsion pair in the sense of \cite[Definition 3.8]{ZH} for extriangulated categories with enough projectives and injectives. In particular, when $\mathcal{C}$ is an exact category with enough projectives and injectives, an ideal $1$-cotorsion pair coincides with a complete ideal cotorsion pair in the sense of \cite[Page 771]{FAHT}.
\end{remark}

Now we give some examples of ideal $n$-cotorsion pairs.

\begin{example}
Let $(\mathcal{C},\mathbb{E},\mathfrak{s})$ be an extriangulated category with enough projective objects and enough injective objects.
Let $[\mathcal{P}]$, $[\mathcal{I}]$, and ${\rm Mor}(\mathcal{C})$ denote the ideal of morphisms factoring through projective objects, the ideal of morphisms factoring through injective objects, and the class of all morphisms in $\mathcal{C}$, respectively.
Then $([\mathcal{P}],{\rm Mor}(\mathcal{C}))$ and $({\rm Mor}(\mathcal{C}),[\mathcal{I}])$ are ideal $n$-cotorsion pairs.
\end{example}

\begin{proof}
We only prove the former assertion, since the proof of the latter is similar.

	(1) On the one hand, for any morphism $i\colon A \ra B \in [\P],\ j\colon C\ra D \in {\rm Mor}(\C)$, $\E^{k}(i,j)=0$ for any integer $1 \leq k \leq n$ implies $[\P]\subseteq \bigcap\limits_{k=1}^{n}{^{\bot_{k}}{\rm Mor}(\C)}$. On the other hand, let $i\colon A \ra B $ be a morphism in $\bigcap\limits_{k=1}^{n}{^{\bot_{k}}{\rm Mor}(\C)}$.
Since $\C$ has enough projective objects, then there exists an $\E$-triangle
	$$K \lra P \stackrel{p}{\lra} B \stackrel{\delta}{\dra},$$
where $P$ is a projective object.	Applying the functor $\Hom(A,-)$ to the above $\E$-triangle, we obtain the following exact sequence
	$$\Hom(A,P) \stackrel{p_{*}}{\lra} \Hom(A,B) \stackrel{\delta_{\sharp}}{\lra} \E(A,K).$$
It follows that	$\E(i,1_{K})=0$ implies $i \in \ker	(\delta_{\sharp})=\im (p_{*})$. Thus we have $i \in [\P]$ and then $\bigcap\limits_{k=1}^{n}{^{\bot_{k}}{\rm Mor}(\C)} \subseteq [\P].$
	
	(2) ${\rm Mor}(\C)=\bigcap\limits_{k=1}^{n}{[\P]^{\bot_{k}}}$ is obvious.
	\vspace{1mm}

	(3) Since $\C$ has enough projectives, we have that $[\P]$ is precovering.
It is obvious that ${\rm Mor}(\C)$ is preenveloping.

This shows that $([\mathcal{P}],{\rm Mor}(\mathcal{C}))$ is an ideal $n$-cotorsion pair.
\end{proof}

In what follows, unless otherwise stated, we assume that $(\C,\E,\s)$ is a Frobenius extriangulated category. Let $\omega$ denote the full subcategory of projective-injective objects, and let $\underline{\C}=\C/[\omega]$ be the stable category associated with $\C$, where $[\omega]$ is the ideal of morphisms factoring through objects in $\omega$. Moreover, $\X$ and $\Y$ are subcategories of $\C$, and $\I$ is an ideal of $\C$ satisfying $\omega \subseteq {\rm Ob}(\I)$.

To prove the main result of this section, we first establish several preliminary facts concerning ideals under the canonical projection from a Frobenius extriangulated category to its stable category. These facts allow us to compare ideal orthogonality in $\C$ with that in $\underline{\C}$, and to relate the precovering and preenveloping properties of ideals in the two categories. Our main result shows that ideal $n$-cotorsion pairs are both preserved and reflected when passing to the stable category.

\begin{lemma} \label{I-precover,I/w-precover}
	\begin{itemize}
		\item [\rm(1)] $\I$ is precovering in $\C$ if and only if $\I/\omega$ is precovering in $\underline{\C}$.
		\item [\rm(2)] $\I$ is preenveloping in $\C$ if and only if $\I/\omega$ is preenveloping in $\underline{\C}$.
	\end{itemize}
\end{lemma}

\begin{proof}
	We just prove that (1) since (2) is the dual of (1).

{\bf Necessity.} For any object M in $\underline{\C}$, there exists an $\I$-precover $f\colon X \ra M$ in $\C$. We shall show that $\underline{f}\colon X \ra M$ is an $\I/\omega$-precover of M in $\underline{\C}$. In fact, for any morphism $\underline{g}\colon X^{\prime}\ra M$ in $\I/\omega$, there exists $g^{\prime}\colon X^{\prime} \ra M \in \I$ such that $\underline{g}=\underline{g^{\prime}}$. It follows that $g-g^{\prime}$ factors through some object $W$ in $\omega$. Assume that $g-g^{\prime}=w_{1}w_{2}$ with $w_{1}\colon W \ra M,\ w_{2}\colon X^{\prime} \ra W $ in $\I$. Then we obtain $g=g^{\prime}+w_{1}w_{2} \in \I$. Therefore there exists $h\colon X^{\prime} \ra X$ such that $g=fh$. Obviously, $\underline{g}=\underline{fh}$.
	
{\bf Sufficiency.} For any object $M$ in $\C$, there exists an $\I/\omega$-precover $\underline{f}\colon X \ra M$ in $\underline{\C}$ and a deflation $p\colon P\ra M$, where $f,\ p \in \I$. We claim that $f^{\prime}=(f\ p)\colon X \oplus P \ra M$ is an $\I$-precover of $M$. In fact, $f^{\prime} \in\I$ since $f^{\prime}=f(1\ 0)+p(0\ 1)$. For any morphism $g\colon X^{\prime} \ra M \in \I$, there is a morphism $\underline{h}\colon X^{\prime} \ra X$ such that $\underline{g}=\underline{fh}$. Since $X=X \oplus P,\ \underline{f}=\underline{f^{\prime}}$ in $\underline{\C}$, there is a morphism $\underline{h^{\prime}}\colon X^{\prime} \ra X\oplus P$ such that $\underline{g}=\underline{f^{\prime}h^{\prime}}$. Then $g-f^{\prime}h^{\prime}$ factors through some object $W \in \omega$, say $g-f^{\prime}h^{\prime}=w_{1}w_{2}$ with $w_{1}\colon W \ra M,\ w_{2}\colon X^{\prime} \ra W$. By Proposition \ref{cor3.16}, $f^{\prime}$ is a deflation, thus there is a morphism $l\colon W \ra X \oplus P$ such that $w_{1}=f^{\prime}l$. It follows that $g=f^{\prime}h^{\prime}+w_{1}w_{2}=f^{\prime}h^{\prime}+f^{\prime}lw_{2}=f^{\prime}(h^{\prime}+lw_{2})$.
\end{proof}

%	\cite[Lemma 4.5]{ZH} shows that if $(\I,\J)$ is an ideal cotorsion pair, then $\I$ is Ext-special precovering if  and only if $\J$ is Ext-special preenveloping. In a triangulated category, every precovering ideal is both Ext-special and $\Hom$-special. Thus, there is no need to distinguish between ideal (co)torsion pair and complete ideal (co)torsion pair in triangulated categories. Especially, an ideal $\I$ is Ext-special precovering if and only if $\I$ is $\Hom$-special precovering if and only if $\I$ is precovering if and only if $(\I,\I{^{\bot_{1}}})$ is a (complete)ideal cotorsion pair if and only if $(\I,\I{^\bot})$ is a (complete)ideal torsion pair \cite{BM}.
%	The conclusion does not hold in a general exact category, but a similar conclusion can be derived by imposing special restrictions on both the category and the ideals in  \cite[Theorem 1.1]{STWZ}.
%	The analogous conclusions also hold in a Frobenius extriangulated category.
%	Especially, \cite[Corollary 1.2]{STWZ} shows that $(\I,\J)$ is a complete ideal cotorsion pair in $\C$ if and only if $(\I/\omega,\J/\omega)$ is a complete ideal cotorsion pair in $\underline{\C}$. Now, we consider the setting of higher ideal extensions.

\begin{lemma} \label{Ext=0 if and only if C=0}
	Let	$\I,\J$ be ideals of $\C$, $\omega$ the subcategory consisting of projective-injective objects. Then $\E^{k}(\I,\J)=0$ in $\C$ if and only if $\underline{\C}(\I/\omega,
	\Sigma^{k}\J/\omega)=0$ in $\underline{\C}$ for any positive integer $k$.
\end{lemma}

\begin{proof}
	Firstly, we prove the case for $k=1$.

{\bf Necessity.} Let $\underline{i}\colon A \ra B,\ \underline{j}\colon X \ra Y$ be morphisms in $\I/\omega$ and $\J/\omega$, respectively.
	Consider the following commutative diagram of triangles in $\underline{\C}$
	\begin{equation}
		\begin{array}{l} \label{diagram of triangles}
			$$\xymatrix{X \ar[r]^{\underline{f}} &E \ar[r]^{\underline{g}} &B \ar[r]^{\underline{h}} &\Sigma X\\
				X \ar[r]^{\underline{l}} \ar@{=}[u] \ar[d]^{\underline{j}} &F \ar[r]^{\underline{m}} \ar[u] \ar[d] &A \ar[r]^{\underline{n}} \ar[u]^{\underline{i}} \ar@{=}[d] &\Sigma X \ar@{=}[u] \ar[d]^{\Sigma\underline{j}}\\
				Y \ar[r]^{\underline{a}} &G \ar[r]^{\underline{b}} &A \ar[r]^{\underline{c}} &\Sigma Y.}$$
			
		\end{array}
	\end{equation}
	
	Claim: there exists a commutative diagram of $\E$-triangles such that the induced commutative of triangles is isomorphic to the diagram (\ref{diagram of triangles}) above.
	
	Take a deflation $p\colon P \ra B$, then $g^{\prime}=(g,p)\colon E \oplus P \ra B$ is also a deflation by Proposition \ref{cor3.16}. Thus we obtain an $\E$-triangle
	$$X^{\prime} \stackrel{f^{\prime}}{\lra} E^{\prime} \stackrel{g^{\prime}}{\lra} B \stackrel{\delta}{\dra},$$
	where $\underline{g^{\prime}}=\underline{g},\ E^{\prime}=E \oplus P$. Therefore, there exists an isomorphic commutative diagram
	$$\xymatrix{X \ar[r]^{\underline{f}} \ar@{-->}[d]_{\cong}  &E \ar[r]^{\underline{g}} &B \ar[r]^{\underline{h}} &\Sigma X \ar@{-->}[d]_{\cong}\\
		X^{\prime} \ar[r]^{\underline{f^{\prime}}}  &E^{\prime} \ar[r]^{\underline{g^{\prime}}} \ar@{=}[u] &B \ar[r]^{\underline{h^{\prime}}} \ar@{=}[u]  &\Sigma X^{\prime}.}$$	
Additionally, by the following diagram
	$$\xymatrix{X^{\prime} \ar[r]^{l^{\prime}} \ar@{=}[d] &F^{\prime} \ar[r]^{m^{\prime}} \ar[d] &A \ar@{-->}[r]^{i^{*}\delta} \ar[d]^{i}&\\
		X^{\prime} \ar[r]^{f^{\prime}} &E^{\prime} \ar[r]^{g^{\prime}} &B \ar@{-->}[r]^{\delta} &,}$$
	we obtain a commutative diagram between triangles
	$$\xymatrix{X \ar[r]^{\underline{l}} \ar@{=}[d] &F \ar[r]^{\underline{m}} \ar[d] &A \ar[r]^{\underline{n}} \ar[d]^{\underline{i}} &\Sigma X\ar@{=}[d]\\
		X \ar[r]^{\underline{f}} \ar[d]^{\cong} &E \ar[r]^{\underline{g}} \ar@{=}[d] &B \ar[r]^{\underline{h}} \ar@{=}[d] &\Sigma X\ar[d]^{\cong}\\
		X^{\prime} \ar[r]^{\underline{f^{\prime}}} \ar@{=}[d] &E^{\prime} \ar[r]^{\underline{g^{\prime}}} &B \ar[r]^{\underline{h^{\prime}}}  &\Sigma X^{\prime}\ar@{=}[d]\\ 	X^{\prime} \ar[r]^{\underline{l^{\prime}}} &F^{\prime} \ar[r]^{\underline{m^{\prime}}} \ar[u] &A \ar[r]^{\underline{n^{\prime}}} \ar[u]_{\underline{i}} &\Sigma X^{\prime}.}$$	
	Therefore we know that $\underline{n^{\prime}}=\underline{h^{\prime}i}\cong\underline{hi}=\underline{n}$ and there exists an isomorphism between triangles
	$$\xymatrix{X^{\prime} \ar[r]^{\underline{l^{\prime}}} \ar[d]_{\underline{x}}^{\cong} &F^{\prime} \ar[r]^{\underline{m^{\prime}}} \ar@{-->}[d]^{\cong} &A\ar[r]^{\underline{n^{\prime}}} \ar@{=}[d]& \Sigma X^{\prime} \ar[d]^{ \Sigma\underline{x}}_{\cong}\\
		X \ar[r]^{\underline{l}}  &F \ar[r]^{\underline{m}} &A \ar[r]^{\underline{n}}  &\Sigma X.}$$
	
For the morphism of $\E$-triangles
	$$\xymatrix{X^{\prime} \ar[r]^{l^{\prime}} \ar[d]^{jx} &F^{\prime} \ar[r]^{m^{\prime}} \ar[d] &A \ar@{-->}[r]^{i^{*}\delta} \ar@{=}[d] &\\
		Y \ar[r]^{a^{\prime}} &G^{\prime} \ar[r]^{b^{\prime}} &A \ar@{-->}[r] &,}$$
	we obtain a commutative diagram between triangles
	$$\small\xymatrix{Y \ar[r]^{\underline{a}}  &G \ar[r]^{\underline{b}} &A \ar[r]^{\underline{c}} \ar@{=}[d]& \Sigma Y \\
		X \ar[r]^{\underline{l}} \ar[u]^{\underline{j}} \ar[d]^{\cong} &F \ar[r]^{\underline{m}} \ar[u]\ar[d]^{\cong} &A \ar[r]^{\underline{n}} \ar@{=}[d] &\Sigma X \ar[u]_{\Sigma\underline{j}}\ar[d]^{\cong}\\
		X^{\prime} \ar[r]^{\underline{l^{\prime}}} \ar[d]_{\underline{jx}} &F^{\prime} \ar[r]^{\underline{m^{\prime}}} \ar[d] &A \ar[r]^{\underline{n^{\prime}}} \ar@{=}[d] &\Sigma X^{\prime}\ar[d]^{\Sigma(\underline{jx})}\\
		Y \ar[r]^{\underline{a^{\prime}}}&G^{\prime} \ar[r]^{\underline{b^{\prime}}} &A \ar[r]^{\underline{c^{\prime}}}&\Sigma Y.}$$
	$\underline{c^{\prime}}=\Sigma(\underline{jx})\underline{n^{\prime}}=\Sigma(\underline{j})\underline{n}=\underline{c}$.
	Thus we have the isomorphism
	$$\xymatrix{Y \ar[r]^{\underline{a}} \ar@{=}[d] &G \ar[r]^{\underline{b}} \ar@{-->}[d]^{\cong} &A\ar[r]^{\underline{c}} \ar@{=}[d]&\Sigma Y \ar@{=}[d]\\
		Y \ar[r]^{\underline{a^{\prime}}}  &G^{\prime} \ar[r]^{\underline{b^{\prime}}}  &A \ar[r]^{\underline{c^{\prime}}}  &\Sigma Y.}$$
	Therefore, the morphism of $\E$-triangles
	$$\xymatrix{X^{\prime} \ar[r]^{f^{\prime}} \ar@{=}[d] &E^{\prime} \ar[r]^{g^{\prime}} &B \ar@{-->}[r]^{\delta} &\\
		X^{\prime} \ar[r]^{l^{\prime}} \ar[d]^{jx} &F^{\prime} \ar[r]^{m^{\prime}}  \ar[u] \ar[d] &A \ar@{-->}[r]^{i^{*}\delta} \ar@{=}[d] \ar[u]^{i}&\\
		Y \ar[r]_{a^{\prime}} &G^{\prime} \ar[r]^{b^{\prime}} &A \ar@{-->}[r]^{(jx)_{*}i^{*}\delta} &}$$
	induces the following commutative diagram of triangles	
	$$\xymatrix{X^{\prime} \ar[r]^{\underline{f^{\prime}}} &E^{\prime} \ar[r]^{\underline{g^{\prime}}} &B \ar[r]^{\underline{h^{\prime}}} &\Sigma X^{\prime}\\
		X^{\prime} \ar[r]^{\underline{l^{\prime}}} \ar@{=}[u] \ar[d]^{\underline{jx}} &F^{\prime} \ar[r]^{\underline{m^{\prime}}} \ar[u] \ar[d] &A \ar[r]^{\underline{n^{\prime}}} \ar[u]^{\underline{i}} \ar@{=}[d] &\Sigma X^{\prime} \ar@{=}[u] \ar[d]^{\Sigma(\underline{jx})}\\
		Y \ar[r]^{\underline{a^{\prime}}} &G^{\prime} \ar[r]^{\underline{b^{\prime}}} &A \ar[r]^{\underline{c^{\prime}}} &\Sigma Y,}$$
	 which is isomorphic to the diagram (\ref{diagram of triangles}).

Since $\E(i,j)=0$, $jx$ factors through $l^{\prime}$ by Proposition \ref{Ext=0 and factorization}. Thus, $\underline{j}$ factors through $\underline{l}$. This means that $\underline{\C}(\underline{i},
\Sigma\underline{j})=0$.

{\bf Sufficiency.} Let $i\colon A \ra B,\ j\colon X \ra Y$ be morphisms in $\I$ and $\J$, respectively. For the morphisms of $\E$-triangles
$$\xymatrix{X \ar[r]^{f} \ar@{=}[d] &E \ar[r]^{g} &B \ar@{-->}[r] &\\
	X \ar[r]^{l} \ar[d]^{j} &F \ar[r]^{m}  \ar[u] \ar[d] &A \ar@{-->}[r] \ar@{=}[d] \ar[u]^{i}&\\
	Y \ar[r]^{a} &G \ar[r]^{b} &A \ar@{-->}[r] &,}$$
consider the induced commutative diagram of triangles
$$\xymatrix{X \ar[r]^{\underline{f}} \ar@{=}[d] &E \ar[r]^{\underline{g}} &B \ar[r]^{\underline{h}} &\Sigma X \ar@{=}[d]\\
	X \ar[r]^{\underline{l}} \ar[d]^{\underline{j}} &F \ar[r]^{\underline{m}}  \ar[u] \ar[d] &A \ar[r]^{\underline{n}} \ar@{=}[d] \ar[u]^{\underline{i}}&\Sigma X \ar[d]^{\Sigma\underline{j}}\\
	Y \ar[r]^{\underline{a}} &D \ar[r]^{\underline{b}} &A \ar[r]^{\underline{c}} &\Sigma Y.}$$

$\underline{\C}(\underline{i},
\Sigma\underline{j})=0$ implies that $\underline{j}$ factors through $\underline{l}$. That is to say, there exists a morphism $\underline{\lambda}\colon F \ra Y$ such that $\underline{j}=\underline{\lambda l}$. Thus $j-\lambda l$ factors through some object $W_{1}$ in $\omega$. Assume that $j-\lambda l=w_{1}w_{2}$ with $w_{1}\colon W_{1} \ra Y,\ w_{2}\colon X \ra W_{1}$. Then there is a morphism $w_{3}\colon F \ra W_{1}$ such that $w_{2}=w_{3}l$ by the property of injective objects.  $j-\lambda l=w_{1}w_{2}=w_{1}w_{3}l$ means that $j$ factors through $l$. Therefore, $\E(i,j)=0$ by Proposition \ref{Ext=0 and factorization}.

When $k=2$, $\E^{2}(\I,\J)=\E(\I,\Sigma\J)$, $\underline{\C}(\I/\omega,
\Sigma^{2}\J/\omega)=\underline{\C}(\I/\omega,
\Sigma(\Sigma\J)/\omega )$. Reverting to the case for $k=1$, we immediately derive that $\E(\I,\Sigma\J)=0$ if and only if $\underline{\C}(\I/\omega,
\Sigma(\Sigma\J)/\omega)=0$. Furthermore, $\E^{2}(\I,\J)=0$ if and only if $\underline{\C}(\I/\omega,
\Sigma^{2}\J/\omega)=0$.

Similarly, we can always reduce higher-order cases to lower-order ones, and ultimately simplify to the case for $k=1$.	
\end{proof}	

\begin{lemma} \label{if and only if}
Let	$\I, \J$ be ideals of $\C$, $\omega$ the subcategory consisting of projective-injective objects. If $\E^{k}(\I,\J)=0$ if and only if $\underline{\C}(\I/\omega,
\Sigma^{k}\J/\omega)=0$ for any positive integer $k$, then the following statements are true.
\begin{itemize}
	\item [\rm(1)] $\I=\bigcap\limits_{k=1}^n{^{\bot_{k}}\J}$ if and only if $\I/\omega=\bigcap\limits_{k=1}^n{^{\bot}}(\Sigma^{k}\J/\omega)$,
	\item [\rm(2)]
	$\J=\bigcap\limits_{k=1}^n{\I^{\bot_{k}}}$ if and only if $\J/\omega=\bigcap\limits_{k=1}^n{(\Omega^{k}\I/\omega)^{\bot}}$.
\end{itemize}
\end{lemma}

\begin{proof}
We just need prove that (1) since (2) is dual.

{\bf Necessity.}  $\underline{\C}(\I/\omega,
\Sigma^{k}\J/\omega)=0$ directly implies $\I/\omega\subseteq\bigcap\limits_{k=1}^n{^{\bot}}(\Sigma^{k}\J/\omega)$. For any morphism $\underline{i}\in\bigcap\limits_{k=1}^n{^{\bot}}(\Sigma^{k}\J/\omega),\ \underline{\C}(\underline{i},
\Sigma^{k}\J/\omega)=0$ shows that $\E^{k}(i,\J)=0$. So we have $i \in \bigcap\limits_{k=1}^n{^{\bot_{k}}\J}=\I$ and then $\underline{i} \in \I/\omega$. That is $\bigcap\limits_{k=1}^n{^{\bot}}(\Sigma^{k}\J/\omega)\subseteq\I/\omega$.

{\bf Sufficiency.}  $\underline{\C}(\I/\omega,
\Sigma^{k}\J/\omega)=0$ implies that $\E^{k}(\I,\J)=0$ which directly shows that  $\I\subseteq\bigcap\limits_{k=1}^n{^{\bot_{k}}\J}$. For any morphism $i \in \bigcap\limits_{k=1}^n{^{\bot_{k}}\J}$, $\E^{k}(i,\J)=0$ implies that $\underline{\C}(\underline{i},
\Sigma^{k}\J/\omega)=0$. Therefore $\underline{i} \in \bigcap\limits_{k=1}^n{^{\bot}}(\Sigma^{k}\J/\omega)=\I/\omega$. Then $i \in \I$ and $\bigcap\limits_{k=1}^n{^{\bot_{k}}\J}\subseteq\I$.
\end{proof}

Now we state and prove our main result.

\begin{theorem} \label{main result of ideal n-cotorsion pair}
Let $(\C,\E,\s)$ be a Frobenius extriangulated category, $\I,\J$ two ideals of $\C$ and $\omega$ the subcategory consisting of projective-injective objects. Then $(\I,\J)$ is an ideal $n$-cotorsion pair in $\C$ if and only if $(\I/\omega,\J/\omega)$ is an ideal $n$-cotorsion pair in $\underline{\C}.$
\end{theorem}	

\begin{proof}
This follows from Lemma \ref{Ext=0 if and only if C=0} and Lemma \ref{if and only if},
Lemma \ref{I-precover,I/w-precover}.
\end{proof}	

As a special case of Theorem \ref{main result of ideal n-cotorsion pair}, we obtain the following result.

\begin{corollary}{\rm\cite[Corollary 1.2]{STWZ}}
Let $(\C,\E,\s)$ be a Frobenius exact category. When $n=1$, $(\I,\J)$ is an ideal cotorsion pair in $\C$ if and only if $(\I/\omega,\J/\omega)$ is an ideal cotorsion pair in $\underline{\C}$.
\end{corollary}

As an application of Theorem~\ref{main result of ideal n-cotorsion pair} in the case $n=1$, we obtain the following result for complete ideal cotorsion pairs.

\begin{corollary}\label{cor58}
Let $(\C,\E,\s)$ be a Frobenius extriangulated category, $(\I,\J)$ an ideal cotorsion pair, and $\omega$ the class of projective-injective objects.
Then the following statements are equivalent.
\begin{itemize}
	\item [\rm(1)] $\I$ is special precovering in $\C$;
	\item [\rm(2)] $\I$ is precovering in $\C$;
	\item [\rm(3)] $\J$ is special preenveloping in $\C$;
	\item [\rm(4)] $\J$ is preenveloping in $\C$;
	\item [\rm(5)] $(\I,\J)$ is a complete ideal cotorsion pair in $\C$;
	\item [\rm(6)] $(\I/\omega,\J/\omega)$ is a complete ideal cotorsion pair in $\underline{\C}$.
\end{itemize}
\end{corollary}

\begin{proof}
(1) $\Rightarrow$ (2) follows from \cite[Proposition 3.7]{ZH}.

(2) $\Rightarrow$ (1)  $\I$ is precovering in $\C$ if and only if $\I/\omega$ is precovering in $\underline{\C}$ by Lemma \ref{I-precover,I/w-precover} if and only if $\I/\omega$ is special precovering in $\underline{\C}$ by \cite[Theorem 5.3.4]{BM}. For every object $C \in \C$, there exists a special $\I/\omega$-precover $\underline{f}\colon X \ra C$ and a homotopy cartesian diagram
$$\xymatrix{A \ar[r] \ar[d]^{\underline{g}} &B \ar[r]^{\underline{ft}} \ar[d]^{\underline{t}} &C \ar[r] \ar@{=}[d] &\Sigma A \ar[d]\\
	Y \ar[r]  &X \ar[r]^{\underline{f}}  &C \ar[r]  &\Sigma Y}$$
with $\underline{g}\in (\I/\omega)^{\bot_{1}}$. It follows from the proof of Lemma \ref{I-precover,I/w-precover} that there exists an $\I$-precover $f^{\prime}\colon X^{\prime}\ra C$ with $\underline{f^{\prime}}=\underline{f}$ and an $\E$-triangle $Y^{\prime} \lra X^{\prime} \stackrel{f^{\prime}}{\lra} C \dra$ such that the induced triangle is isomorphic to the triangle $Y \lra X \stackrel{\underline{f}}{\lra} C \lra \Sigma Y$, i.e.,
$$\xymatrix{Y^{\prime} \ar[r] \ar[d]^{\cong} &X^{\prime} \ar[r]^{\underline{f^{\prime}}} \ar@{=}[d] &C \ar[r] \ar@{=}[d] &\Sigma Y^{\prime} \ar[d]^{\cong}\\
	Y \ar[r] &X \ar[r]^{\underline{f}} &C \ar[r] &\Sigma Y,
}$$
where $X^{\prime}=X\oplus P$. Note that $X^{\prime} =X$ in $\underline{\C}$.
For the composition $t_{1}\colon B \stackrel{t}{\ra} X \ra X^{\prime}$, there is a deflation $p_{1}\colon P_{1} \ra X^{\prime}$ such that $t^{\prime}=(t_{1}\ p_{1})\colon B^{\prime}=B \oplus P_{1} \ra X^{\prime}$ is a deflation, then we obtain a commutative diagram
$$\xymatrix{A^{\prime} \ar[r] \ar@{-->}[d]^{k^{\prime}} &B^{\prime} \ar[r]^{f^{\prime}t^{\prime}} \ar[d]^{t^{\prime}} &C \ar@{-->}[r]^{\delta^{\prime}} \ar@{=}[d]& \\
	Y^{\prime} \ar[r]  &X^{\prime} \ar[r]^{f^{\prime}}  &C \ar@{-->}[r]^{\delta} &.}$$
Thus, this yields an induced commutative diagram between triangles
$$\xymatrix{A^{\prime} \ar[r] \ar[d]^{\underline{k^{\prime}}} &B^{\prime} \ar[r]^{\underline{f^{\prime}t^{\prime}}} \ar[d]^{\underline{t^{\prime}}=\underline{t}} &C \ar[r] \ar@{=}[d] &\Sigma A^{\prime} \ar[d]^{\cong}\\
	Y^{\prime} \ar[r] &X^{\prime} \ar[r]^{\underline{f^{\prime}}=\underline{f}} &C \ar[r] &\Sigma Y^{\prime}.
}$$

Therefore, $\underline{k^{\prime}}\cong \underline{g} \in (\I/\omega)^{\bot_{1}}$. Finally, for any $i\in \I$, $\underline{\C}(\underline{i},\Sigma\underline{k^{\prime}})=\E(\underline{i},\underline{k^{\prime}})=0$ in $\underline{\C}$. Thus $\E(i,k^{\prime})=0$ by Lemma \ref{Ext=0 if and only if C=0}, which implies $k^{\prime}\in \I^{\bot_{1}}$.
This shows that $\I$ is special precovering.

$(3)\Leftrightarrow(4)$ is similar to $(1)\Leftrightarrow(2)$.

$(1)\Leftrightarrow(3)$ follows from \cite[Lemma 4.5]{ZH}.

$(1)\Leftrightarrow(5)$ is obvious by the definition.

$(5)\Leftrightarrow(6)$ follows from Theorem \ref{main result of ideal n-cotorsion pair} when $n=1$.
\end{proof}

\begin{remark}
Recall that every Frobenius exact category can be regarded as a Frobenius extriangulated category. Thus, Corollary~\ref{cor58} recovers \cite[Theorem~1.1]{STWZ}.
\end{remark}

\textbf{Yixia Zhang}\\
School of Mathematics and Statistics, Changsha University of Science and Technology, 410114 Changsha, Hunan,  P. R. China\\
E-mail: yxzhangmath@163.com\\[0.3cm]
\textbf{Panyue Zhou}\\
School of Mathematics and Statistics, Changsha University of Science and Technology, 410114 Changsha, Hunan,  P. R. China\\
E-mail: panyuezhou@163.com

\end{document}